\documentclass[a4paper,11pt]{article}

\usepackage{amsmath,amscd,amssymb,amsthm,amssymb}

\usepackage{cite}
\usepackage{bibgerm}

\usepackage[super]{nth}

\usepackage[british, USenglish]{babel}
\usepackage{tikz-cd}

\usepackage{fullpage}

\usepackage{pxfonts}
\newtheorem{thm}{Theorem}[section]
\newtheorem{prop}[thm]{Proposition}
\newtheorem{lem}[thm]{Lemma}
\newtheorem{cor}[thm]{Corollary}
\theoremstyle{definition}
\newtheorem{rem}[thm]{Remark}

\newtheorem{exmpl}[thm]{Example}

\newcommand{\e}{\varepsilon}

\newcommand{\SOg}{{\rm{\bf SO}}}

\newcommand{\End}{{\rm {\bf End}}}

\newcommand{\SUg}{{\rm {\bf SU}}}

\newcommand{\Hom}{{\rm {\bf Hom}}}

\newcommand{\Sym}{{\rm {\bf Sym}}}

\newcommand{\NN}{\mathbb N}
\newcommand{\HH}{\mathbb H}
\newcommand{\ZZ}{\mathbb Z}
\newcommand{\RR}{\mathbb R}
\newcommand{\CC}{\mathbb C}

\newcommand{\PP}{\mathbb P}

\newcommand{\im}{{\rm im\,}}

\newcommand{\Sq}{{\rm Sq}}

\title{A note on the topology of irreducible $\SOg(3)$-manifolds}
\author{Panagiotis Konstantis\footnote{Phillips Universit\"at Marburg}}
\date{}

\begin{document}

\maketitle

\begin{abstract}
A connected, oriented $5$-manifold $M$ is an \emph{irreducible $\SOg(3)$-manifold} if there exists a rank $3$ vector bundle $\eta$ over $M$ such that the tangent bundle is isomorphic to the bundle of symmetric trace-free endomorphisms of $\eta$.
In this article we give necessary and sufficient topological conditions for the existence of irreducible $\SOg(3)$-manifolds. There we have to distinguish if the manifold is spin or not. We also provide some new examples of irreducible $\SOg(3)$-manifolds. 
\end{abstract}

\section{Introduction}\label{SIntroduction}
Throughout this article let $M$ be always a smooth, connected and oriented manifold. 
In this paper our aim is to determine topological conditions for the existence of certain rank $3$ subbundles of the tangent bundle for $5$-manifolds. A related situation (cf. Remark \ref{RReductionGeometrically}) was studied by E. Thomas in \cite{Thomas1967a}. There he investigated the conditions to the existence of \emph{two-fields} on manifolds $M$ of dimension $4k+1$ ($k>0$). A \emph{two-field} is a pair of tangent vector fields which are linearly independent in every point of the manifold. Hence every two-field determines a trivial rank $2$ subbundle of the tangent bundle of $M$. He proved the following

\begin{thm}[{\cite[Corollary 1.2]{Thomas1967a}}]\label{TThomasStandardStructure}
Let $M$ be a closed, connected,  spin manifold of dimension $4k+1$, $k>0$. Then $M$ admits a two--field if and only if $w_{4k}(M)=0$ and $\hat\chi(M)=0$,
\end{thm}
\noindent where $w_{4k}(M)$ is the $4k$-th Stiefel-Whitney class of $M$ and $\hat\chi(M)$ is the \emph{semi-characteristic} of $M$, which is defined as
\[
 \hat\chi(M)=\sum_{i=0}^{2k} \dim_{\ZZ_2}H^i(M;\ZZ_2) \mod 2.
\]
If $M$ is a spin manifold like in Theorem \ref{TThomasStandardStructure} with $\dim M =5$, we can apply Wu's formula to see that $w_4(M)=w_2(M)\smile w_2(M)=0$, hence we obtain
\begin{cor}[see also \cite{Thomas1968}]\label{CThomasStandStructure}
Suppose $M$ is a closed, connected spin $5$-manifold. Then $M$ admits a two--field if and only if $\hat\chi(M)=0$.
\end{cor}
Later, M. Atiyah proved a more general theorem using K-theory, where $M$ has not to be spin. 
\begin{thm}[{\cite[Theorem 5.1]{Atiyah1970}}]\label{TAtiyah}
Let $M$ be a connected, closed, oriented manifold of dimension $4k+1$, $k>0$. Then $M$ admits a two--field if and only if $k(M)=0$, where 
\[
 k(M) = \sum_{i=0}^{2k} \dim_\RR H^{2i}(X;\RR) \mod 2
\]
is the Kervaire semi-characteristic of $M$.
\end{thm}
For an orientied $5$--manifold $M$, both semi--characteristics are connected by the \emph{Lusztig--Milnor--Peterson formula}
\[
k(M) - \hat\chi(M)  = \langle w_2(M) \smile w_3(M), [M]\rangle,
\]
where $[M] \in H_5(M)$ is the fundamental class (cf. \cite{LMP1969}).

The existence of a two--field on a $5$--manifold $M$ could be rephrased to the existence of a rank $3$ subbundle $\eta$ of $\tau_M$, such that the Whitney sum of $\eta$ with a trivial rank $2$ vector bundle is isomorphic to $\tau_M$. In this article we would like to study a similar situation. A manifold $M$ of dimension $5$ is called an \emph{irreducible $\SOg(3)$--manifold} if there is a rank $3$ vector bundle $\eta$ over $M$ such that the tangent bundle of $M$ is isomorphic to the bundle of symmetric trace--free endomorphisms of $\eta$ (see Remark \ref{RReductionGeometrically} for an explanation why such manifolds are called \emph{irreducible}). We will prove two main theorems about the topology of irreducible $\SOg(3)$-manifolds, which may be summarized as follows:

%
\begin{thm}\label{TItroduction}
Let $M$ be a closed, oriented and connected manifold of dimension five. 
  \begin{enumerate}
   \item[(a)] Suppose $w_2(M)=0$ (i.e. if $M$ is spin). Then $M$ is an irreducible $\SOg(3)$-manifold if and only if 
   \begin{enumerate}
    \item[(i)] $w_4(M)=0$ and the first Pontryagin class $p_1\in H^4(M;\ZZ)$ is divisible by five,
    \item[(ii)] $\hat\chi(M)=0$.
   \end{enumerate}

   \item[(b)] Suppose $w_2(M)\neq 0$ and $H^4(M;\ZZ)$ contains no element of order $4$. Then $M$ is an irreducible $\SOg(3)$-manifold if and only if $w_4(M)=0$ and $p_1(M)$ is divisible by five.
  \end{enumerate}
\end{thm}
For a simply connected, closed $5$--manifold we have $H^4(M;\ZZ)=H^4(M;\ZZ_2)=0$ by Poincar\'e duality and $\hat\chi(M)= 1 + \dim_{\ZZ_2} H_2(M;\ZZ_2) \mod 2$. Hence we obtain

\begin{cor}\label{CNonSpinCase}Let $M$ be a simply connected, closed $5$--manifold.
\begin{enumerate}
 \item[(a)] Let $w_2(M)=0$. Then $M$ is an irreducible $\SOg(3)$--manifold if and only if $\dim_{\ZZ_2}H_2(M;\ZZ_2)$ is odd.
 \item[(b)] Let $w_2(M)\neq 0$. Then $M$ is an irreducible $\SOg(3)$--manifold.
\end{enumerate}
\end{cor}

Furthermore from Corollary \ref{CThomasStandStructure} we obtain that Theorem \ref{TItroduction} (a) is equivalent to

\begin{cor}\label{CT14aBetterFormulated}
Let $M$ be a closed spin $5$--manifold. Then $M$ is an irreducible $\SOg(3)$--manifold if and only if $M$ admits a two--field and $p_1(M)$ is divisible by five.
\end{cor}


%
%
%

\begin{rem}\label{RReductionGeometrically}
%
We would like to give an alternative definition of irreducible $\SOg(3)$--manifolds, which will be needed for proving Theorem \ref{TItroduction} (a).  Let $G$ be a Lie group and $\rho \colon G \to \SOg(n)$ an embedding of $G$ as a Lie subgroup of $\SOg(n)$. Let $\tau_M$ be the tangent bundle of $M$. A \emph{$G$-structure on $\tau_M$} is a reduction of the $\SOg(n)$-principal bundle of $\tau_M$ to a $G$-principal bundle over $M$ (where $n$ is the dimension of $M$). If we regard $\tau_M$ as a map from $M$ to the classifying space $B\SOg(n)$, this definition is equivalent to the existence of a lift $\hat\tau_M$ for $\tau_M$ such that the following diagram commutes up to homotopy
\begin{center}
 \begin{tikzcd}
   & BG \arrow{d}{B\rho} \\
   M\arrow{ru}{\hat\tau_M} \arrow{r}{\tau_M} & B\SOg(n), \\
 \end{tikzcd}
\end{center}
where $B\rho$ is the map induced by $\rho$ on the classifying spaces. The existence of such a lift can be decided with obstruction theory, where sometimes the obstructions can be expressed in terms of characteristic classes of $M$. 

The existence of a two-field is equivalent to the existence of a $\SOg(n-2)$--structure on $M$ with respect to the standard embedding of $\SOg(n-2)$ into $\SOg(n)$ ($n=\dim M$). We call such a structure a \emph{standard $\SOg(n-2)$--structure on $M$}.

For $n=5$ there is another embedding of $\SOg(3)$ into $\SOg(5)$ besides the standard one (see \cite{ABBF11}):  Identify $\RR^5$ as a vector space with the space $\Sym_0(\RR^3)$ of symmetric trace free endomorphisms of $\RR^3$. Then for $h \in \SOg(3)$ and $X \in \Sym_0(\RR^3)$ we set $\rho(h)X:= hXh^{-1}$. This defines the \emph{irreducible embedding}
\[
 \rho \colon \SOg(3) \to \End(\RR^5),\quad h\mapsto \rho(h).
\]
It is now easy to see that $\rho(h)$ preserves the standard metric on $\RR^5$, which makes $\rho$ a map into $\SOg(5)$. It is clear that, as representations, the standard embedding and the irreducible embedding of $\SOg(3)$ into $\SOg(5)$ are not equivalent.
An irreducible $\SOg(3)$--manifold is an $\SOg(3)$--structure with respect to the embedding $\rho$.
\end{rem}

{\bf Historical remark.} First steps were made in \cite{Bobienski2006} and later in \cite{ABBF11}. The author of \cite[Theorem 1.4]{Bobienski2006} claimed that an irreducible $\SOg(3)$-structure exists if and only if the manifold admits a standard $\SOg(3)$-structure and the first Pontryagin class is divisible by $5$. However in \cite[Example 3.1]{ABBF11} it was shown that the symmetric space $\SUg(3)/\SOg(3)$ does not have a standard $\SOg(3)$-structure, but it admits an irreducible $\SOg(3)$-structure. Nevertheless in \cite{Bobienski} M. Bobienski reports that the proof of Theorem 1.4 in \cite{Bobienski2006} should work if one assumes the manifold is spin. Indeed we will prove in section \ref{SPostnikovTower} that this is true, but we will use a different approach as in \cite{Bobienski2006} (Bobienski tries to compare the Moore-Postnikov towers of the irreducible and the standard representation of $\SOg(3)$, where in this article we apply the methods of E. Thomas directly to this special case). In \cite[Theorem 3.2]{ABBF11} the authors prove some necessary conditions for the existence of an irreducible $\SOg(3)$-structure on a $5$-manifold, using a special characterization of the tangent bundle. 

\bigskip
In section 2 we will prove Theorem \ref{TItroduction} (a). In particular we will determine properties of the obstructions for lifting a map $\xi   \colon M\to B\SOg(5)$ to a map $M \to B\SOg(3)$ through the fibration $B\rho \colon B\SOg(3) \to B\SOg(5)$. In particular we will use the techniques of E. Thomas, see \cite{Thomas1966},\cite{Thomas1967} and \cite{Thomas1967a}. The second part of Theorem \ref{TItroduction} will be proved in section \ref{SNonSpinManifolds}, which is a quite short proof and depends heavily on the work of \u{C}adek and Van\v{z}ura, cf. \cite{Cadek1993}. In that article the authors classify $3$- and $5$-dimensional vector bundles over certain $5$-complexes. Note that we rely on the condition on $H^4(M;\ZZ)$, since the classification of vector bundles is given by characteristic classes (see e.g. \cite{Thomas1968}), where this condition is necessary (see also Remark \ref{RDrawback}).


The contents of the last section are examples of irreducible $\SOg(3)$--manifolds, where we will exploit Theorem \ref{TItroduction}. In particular: the symmetric space $Y:=\SUg(3)/\SOg(3)$ does not admit a standard $\SOg(3)$-structure (since $k(Y)=1$, see Theorem \ref{TAtiyah}), but an irreducible one with $\hat\chi(Y)=0$. So it was conjectured in \cite{ABBF11}, that $\hat\chi(M)$ could be an obstruction to the existence of an irreducible $\SOg(3)$--structure. However we will prove 

\begin{prop}\label{PCounterExampleToFriedrich}
There is a compact and connected irreducible $\SOg(3)$--manifold $M$ with $\hat\chi(M)=1$. 
\end{prop}
\noindent
The manifold mentioned in the proposition above is constructed as the total space of a circle bundle over a simply connected, compact $4$--manifold. We will also show

\begin{prop}\label{PConnectedSumsOfSpin5Manifolds}
The connected sum of an odd number of irreducible $\SOg(3)$--manifolds which are spin is an irreducible $\SOg(3)$--manifold.
\end{prop}

With proposition we will be able to consider new examples of irreducible $\SOg(3)$-manifolds.


%

\section*{Acknowledgement}
The author would like to thank an unknown reviewer for his helpful comments and for improving the structure of this article. Moreover he thanks Benjamin  Antieau for his help to clarify Remark \ref{ROnTheExistenceOf3BundlesOver5Complex}.

\section{Preliminaries for the Proof of Theorem 1.4 (a)}\label{SPostnikovTower}
We will make use of the definitions of Remark \ref{RReductionGeometrically}.
Let $\rho \colon \SOg(3) \to \SOg(5)$ be the $5$-dimensional irreducible representation of $\SOg(3)$ and $B\rho \colon B_5 \to B_3$ the induced map on classifying spaces where $B_k:=B\SOg(k)$. Let $M$ be an oriented, compact $5$-manifold with $w_2(M)=0$ and $\tau_M \colon M \to B_5$ the classifying map for the tangent bundle of $M$. Consider the following diagram
\begin{center}
 \begin{tikzcd}
   & X=\SOg(5)/\rho(\SOg(3)) \arrow{d} &\\
   & B_3 \arrow{dd}{B\rho} \arrow{dr}{q} &\\
   &   & E\arrow{dl}{p}\\
   M\arrow{ruu}{\hat\tau_m} \arrow{r}{\tau_M} & B_5 & \\
 \end{tikzcd}
\end{center}
The manifold $M$ admits an irreducible $\SOg(3)$--structure if and only if there is a map $\hat\tau_M \colon M \to B_3$ such that $B\rho \circ\hat\tau_M$ is homotopic to $\tau_M$. So Theorem \ref{TItroduction} (a) is implied by the following 
\begin{lem}\label{LKeyLemma}
There is a topological space $E$ and fibration $p \colon E \to B_5$ such that
\begin{enumerate}
\item[(a)] The map $\tau_M \colon M \to B_5$ lifts to a map $M \to E$ if and only if $\rho_5(p_1(M))=0$ and $w_4(M)=0$, where $\rho_5 \colon H^*(M;\ZZ) \to H^*(M;\ZZ_5)$ is the mod $5$ reduction of integer cohomology classes.
\item[(b)] There is a fibration $q \colon B_3 \to E$ such that a lift $\eta \colon M \to E$ of $\tau_M$ lifts to a map $M \to B_3$ if and only if $\hat\chi(M) =0$ (this condition depends only on $M$ not on $\eta$).
\end{enumerate}
\end{lem}
\begin{proof}
We start the definition of $E$ and the fibration $p \colon E \to B_5$. Therefore we need to compute $\pi_3(X)$. Consider the long exact homotopy sequence for $\SOg(3) \to \SOg(5) \to X$ and the fact the $\rho$ induces an isomorphism on the first homotopy groups of $\SOg(3)$ and $\SOg(5)$. By the Hurewicz theorem and the fact that $H_3(X)=\ZZ_{10}$ one obtains $\pi_3(X) \cong \ZZ_{10}$.
Assume now coefficients in $\pi_3(X) \cong \ZZ_{10}\cong \ZZ_5\oplus \ZZ_2$. The Serre exact sequence \cite[Proposition 3.2.1]{Kochm1996} for the fibration $X \to B_3 \to B_5$ is
\[
H^3(X) \stackrel{\tau}{\longrightarrow} H^4(B_5) \stackrel{B\rho^*}{\longrightarrow}H^4(B_3).
\]
where $\tau$ is the transgression map of the fibration $X \to B_3 \to B_5$. By the Universal Coefficient Theorem we have $H^3(X;\pi_3(X)) \cong \Hom(\pi_3(X),\pi_3(X))$. Let $\gamma_1 \in H^3(X;\pi_3(X))$ be the element which corresponds to the identity in $\Hom(\pi_3(X),\pi_3(X))$. We consider the element $-\tau(\gamma_1) \in H^4(B_5;\pi_3(X)) \cong H^4(B_5;\ZZ_{10})$ as a map $-\tau(\gamma_1) \colon B_5 \to K(\ZZ_{10},4)$. Define the space $E$ as the pullback of the pathspace fibration $\Omega K(\ZZ_{10},4) \to \mathcal P \to K(\ZZ_{10},4)$ by $-\tau(\gamma_1)$. Hence one obtains a fibration $p \colon E \to B_5$ with homotopy fibre $\Omega  K(\ZZ_{10},4)$.\\
\\
We prove now \emph{(a)}. It suffices to prove that following conditions are equivalent:
\begin{enumerate}
\item[(i)] the map $\tau_M \colon M \to B_5$ lifts to a map $M \to E$,
\item[(ii)] $\tau^*_M\tau(\gamma_1) =0$,
\item[(iii)] $\tau^*_M\ker(B\rho^*) =0$,
\item[(iv)] $\tau_M^*\rho_5(p_1)=0$ and $\tau^*_M w_4=0$,
\item[(v)] $\rho_5(p_1(M))=0$ and $w_4(M)=0$,
\end{enumerate}
where $p_1$ is the first Pontryagin class and $w_4$ the fourth Stiefel--Whitney class of the universal bundle oder $B_5$.

The equivalence of \emph{(i)} and \emph{(ii)} is essentially proved in \cite[p.3]{Thomas1966}, since the primary obstructions $\mathfrak o_E$ and $\mathfrak o_{B_3}$ to lifting a map $\xi \colon M \to B_5$ to a map $M \to E$ and to a map $M \to B_3$ are equal, and because $\mathfrak o_E$ is complete (since the homotopy fibre of $\Omega K(\ZZ_{10},4)$ of $p$ is $4$--connected).

Since $\im \tau = \ker(B\rho^*)$ and $H^3(X)$ is generated by $\gamma_1$, we obtain that \emph{(ii)} is equivalent to \emph{(iii)}.

Next we compute $\ker(B\rho^*) \subset H^4(B_4;\ZZ_{10}) =  H^4(B_4;\ZZ_{5})\oplus  H^4(B_4;\ZZ_{2})$. It is known that $H^4(B_5;\ZZ) = \ZZ[p_1]\oplus 2\text{-\emph{torsion}}$ and $H^3(B_5;\ZZ)=H^5(B_5;\ZZ)=0$. Moreover the $2$-torsion is of order $2$. The long exact sequence associated to 
    \[
      0 \longrightarrow \ZZ \stackrel{\cdot 5}{\longrightarrow} \ZZ \longrightarrow \ZZ_5 \longrightarrow 0
    \]
    yields the short exact sequence
    \[
      0 \longrightarrow 5\cdot H^4(B_5;\ZZ) \longrightarrow H^4(B_5;\ZZ) \longrightarrow H^4(B_5;\ZZ_5) \longrightarrow 0.
    \]
     Furthermore $5\cdot H^4(B_5;\ZZ) = 5\ZZ[p_1] \oplus 2\text{-\emph{torsion}}$, hence 
    \[
    H^4(B_5;\ZZ_5)=\ZZ_5[\rho_5(p_1)].
    \]
Let $\rho^T$ be $\rho$ restricted to the maximal Torus of $\SOg(3)$, then this map is given by $\rho^T \colon S^1 \to S^1 \times S^1$, $z\mapsto (z,z^2)$ (cf. \cite[p. 69]{ABBF11}). Hence $B\rho^*(p_1)=10p_1$ thus $B\rho^*(\rho_5(p_1))=0$. Finally we would like to determine the kernel of $B\rho^*$ with coefficients in $\ZZ_2$. Again, it is known that $H^4(B_5;\ZZ_2)= \ZZ_2[w_2^2,w_4]$. Using \cite[Theorem 5.9]{Mimura1991} and the explicit map $\rho \colon \SOg(3) \to \SOg(5)$ given in \cite[p.68]{ABBF11} we compute $B\rho^*(w_2)=w_2$ and $B\rho^*(w_4)=0$. It follows that $\ker(B\rho^*)$ is generated by $\rho_5(p_1)$ and $w_4$ and this shows the equivalence \emph{(iii)} and \emph{(iv)}.

From the naturality property of charactersitic classes we have $\tau^*_M p_1 = p_1(M)$ and $\tau^*_M w_4 =w_4(M)$ which shows that \emph{(iv)} is equivalent to \emph{(v)}. This proves part \emph{(a)}.

We proceed with the proof of part \emph{(b)}. First note that $B\rho^*(-\tau(\gamma_1))=0$ since $-\tau(\gamma_1)$ lies in the image of $\tau$. This means that $B\rho \circ (-\tau(\gamma_1)) \colon B_5 \to K(\ZZ_{10},4)$ is homotopic to a constant map, hence there is a lift $q \colon B_3 \to E$ of $B\rho$ through $p$ (cf. \cite[p.3]{Thomas1966}). We replace $q$ by a map which is a fibration and homotopic to $q$. Denote this fibration again by $q$. Let $F$ be its homotopy fibre. From \cite[p.189]{Thomas1967} we know that $F$ is $3$--connected with $\pi_4(F) \cong \pi_4(X)$. From Lemma 3.2 in \cite{ABBF11} and the long exact homotopy sequence it follows that $\pi_4(X)\cong \ZZ_2$. Let $\gamma_2 \in H^4(F;\pi_4(F)) = \Hom(\pi_4(F),\pi_4(F))$ be the identity element as before. We assume from now on coefficients in $\pi_4(F) \cong \ZZ_2$. The Serre exact sequence \cite[Proposition 3.2.1]{Kochm1996} for the fibration $F \to B_3 \to E$ is
\[
H^4(F) \stackrel{\tau}{\longrightarrow} H^5(E)\stackrel{q^*}{\longrightarrow} H^5(B_3).
\]
Suppose that $\eta \colon M \to E$ is a lift of $\tau_M \colon M \to B_5$. From \cite[p.190]{Thomas1967} we deduce that $\eta$ lifts to a map $M \to B_3$ if and only if $\eta^*\tau(\gamma_2)=0$. Hence it remains to prove there exists a lift $\eta\colon M\to E$ of $\tau_M$ such that the class $\eta^*\tau(\gamma_2)$ is mapped to $\hat\chi(M)$ under the isomorphism $H^5(M) \cong \ZZ_2$. We identify $H^5(M)$ with $\ZZ_2$. We will show
\begin{equation*}
 \hat\chi(M) = I_2(M) = t_5^{-1} \left(\Omega (U_M)\right)= \eta^*\tau(\gamma_2). \tag{$\ast$}
\end{equation*}
In the lines which follow we introduce the objects in $(\ast)$ and prove the equalities. 

The number $I_2(M)$ is defined in \cite[p. 89]{Thomas1967a} which is a mod $2$ index of a two-field with singularities on $M$. The first equality is a particular case of \cite[Theorem 1.1]{Thomas1967a}.

For the second equality in $(\ast)$ let $T$ be the Thom space of $\tau_M$ and $U_M \in H^5(T;\ZZ)$ shall denote the Thom class of $T$. Denote by $t_j \colon H^j(M) \to H^{5+j}(T)$ the Thom isomorphism. Furthermore  recall first the Adem relation
\[
\Sq^2\Sq^4 + \Sq^5\Sq^1 = \Sq^6.
\]
Thus on integral classes of dimenesion $\leq 5$ we obtain
\[
\Sq^2\Sq^4 =0.
\] 
Let $\Omega$ be the secondary cohomology operation associated to the above relation applied to the cohomology of the space $T$, i.e. $\Omega$ is defined on $u \in H^j(T)$ ($j\leq 5$) provided $\Sq^4 u=0$ and its image is a coset in $H^{5+j}(T)$ by the subgroup $\Sq^2 H^{5+j-2}(T)$.
Note first, that $\Omega$ is defined on $U_M$ since $\Sq^4 U_M = t_4(w_4(M))$ and since $\eta$ is a lift of $\tau_M \colon M \to B_5$  we have $w_4(M)=0$ by part $(a)$ of this lemma. Second, by Poincar\'{e} duality the Wu class of $\Sq^2$ is $w_2(M)$, hence $Sq^2 H^8(T)=0$ since we assume $w_2(M)=0$ and therefore $\Omega$ is single--valued. The second equality of $(\ast)$ follows now from Theorem 2.2 in \cite{Thomas1967a}.

The last equality in $(\ast)$ can be shown as follows: Let $\gamma_5$ denote the universal vector bundle of $B_5$ and $T$ its Thom space and $U$ the Thom class of $\gamma_5$. If $Y$ is a topological space and $g \colon Y \to B_5$ is a continuous map, then we denote by $T_Y$ and $U_Y$ the Thom space and the Thom class of $g^*\gamma_5$. Set $w:=w_4 \in H^4(B_5;\ZZ_2)$ and $\alpha := \Sq^2$. Then the $w$ is realizable in the sense of $(6.1)$ in \cite[p. 102]{Thomas1967a} and $(w,\alpha)$ is \emph{admissible}, see \cite[p.102, (6.2), (6.3)]{Thomas1967a} and see \cite[p.108]{Thomas1967a} for a proof that $(w,\alpha)=(w_4,\Sq^2)$ is admissible. Let $\tilde p \colon \tilde E \to B_5$ denote the principal fibration induced by $w=w_4$ considered as a map $B_5 \to K(\ZZ_2,4)$. From Theorem 6.4 in \cite{Thomas1967a} we have that there is an $m \in H^5(B_5)$ and a $k \in H^5(\tilde E)$ such that
\[
U_{\tilde E} \smile (k + \tilde p^*m) \in \Omega(U_{\tilde E}).
\]
One computes that $H^5(\tilde E) \cong \ZZ_2$ and $k$ is the generator.
The inclusion $K(\ZZ_2,4) \to K(\ZZ_{10},4) \cong K(\ZZ_2,4) \times K(\ZZ_5,4)$ (which is on the first factor the identity and on the second it is the constant map to the basepoint of $K(\ZZ_5,4)$) induces a map $f\colon E \to \tilde E$ such that $f^*(k) = \tau(\gamma_2)$. Hence by the naturality of $\Omega$ one obtains
\[
 U_{E} \smile (\tau(\gamma_2) + p^*m) \in \Omega(U_{E}).
\]
Using again the naturality of $\Omega$ with respect to the map $\eta \colon M \to E$ we get
\[
 U_{M} \smile (\eta^*\tau(\gamma_2) + \eta^*p^*m) = \Omega(U_{M}).
\]
Observe now that $H^5(B_5)$ is generated by $w_2\smile w_3$ and $w_5$. Since $w_2(M)=0$ and $w_5(M) = \Sq^1 w_4(M)=0$ it follows that $\eta^*p^*m=0$. Hence $U_{M} \smile \eta^*\tau(\gamma_2)) = \Omega(U_{M})$, i.e. $\eta^*\tau(\gamma_2) = t_5^{-1}(\Omega(U_M))$, which proves the last equality and finally part $(b)$.
\end{proof}
\bigskip
\bigskip
In Theorem \ref{TItroduction} (a) we saw that an irreducible $\SOg(3)$-manifold $M$ with $w_2(M)=0$ possesses also a standard structure. Hence there are two different descriptions of the tangent bundle of $M$. We close this section with a proposition which compares these two descriptions. In Lemma 3.1 of \cite{ABBF11} the existence of an irreducible $\SOg(3)$-structure is equivalent to the existence of a $3$-dimensional vector bundle $\xi$ over $M$ such that the tangent bundle $\tau_M$ is isomorphic to the symmetric trace-free endomorphisms $\Sym_0(\xi)$ of $\xi$. On the other side, since $M$ admits also a standard $\SOg(3)$-structure, the tangent bundle is isomorphic to $\eta\oplus \e^2$, where $\eta$ is a rank $3$ vector bundle and $\e^2$ the trivial rank 2 vector bundle over $M$. Finally to formulate the proposition, we have to introduce a certain kind of operation on rank $3$ vector bundles over $5$-manifolds:

Let $\xi$ be a $3$-dimensional vector bundle with a spin structure over $M$. Since $M$ is of dimension $5$, we have $ \xi \in [M,\HH\PP^\infty] \cong [M,S^4]$. Let $g \colon S^4 \to S^4$ be a map of degree $5$. Then $g$ induces a map $G \colon [M,S^4] \to [M,S^4]$, such that $G(\xi)=g\circ \xi$. We denote by the same letter $G$ the induced map on $[M,\HH\PP^\infty]$. It is clear that $G$ depends only on the homotopy class of $g$.

\begin{prop}\label{PRelationEandV}
Let $M$ be a spin $5$-manifold such that $H^4(M;\ZZ)$ has no $2$-torsion. Suppose furthermore that $M$ admits an irreducible $\SOg(3)$-structure. Let $\xi$ be $3$-dimensional vector bundle with spin structure such that 
\[
  \tau_M \cong \Sym_0^2(\xi).
\]
Then we have 
\[
  \tau_M \cong G(\xi) \oplus \varepsilon^2.
\]
\end{prop}

\begin{proof}
It is known that $H^*(\HH\PP^\infty;\ZZ) = \ZZ[u]$ where $u \in H^4(\HH\PP^\infty;\ZZ)$. Moreover we have $G^*u=5u$ by definition and $p_1 = 4u$ where $p_1$ generates the algebra $H^*(B\SOg(3);\ZZ)/{\rm torsion}$. Therefore by naturality we obtain $p_1(G(\xi))=5p_1(\xi)=p_1(M)$ and $w_2(G(\xi))=0$. By Lemma 1 of \cite{Thomas1968} we have that $\tau_M$ and $G(\xi)\oplus\varepsilon^2$ are stably isomorphic and by Lemma 3 of the same paper we obtain that they are isomorphic, since $M$ admits a standard $\SOg(3)$-structure.
\end{proof}

\section{Proof of Theorem \ref{TItroduction} (b)}\label{SNonSpinManifolds}

In this section we work under the assumptions of Theorem \ref{TItroduction} (b) and let $\mathcal P \colon H^2(M;\ZZ_2) \to H^4(M;\ZZ_4)$ denote the Pontryagin square. Under these conditions the authors of \cite{Cadek1993} prove the following:

\begin{thm}\label{TCV5Bundles}
Let $\xi$ be a vector bundle of rank $5$ over $M$. Then $\xi$ is uniquely determined by $w_2(\xi), w_4(\xi)$ and $p_1(\xi)$ such that
\[
  \rho_4\,  p_1(\xi) = \mathcal P w_2(\xi) + i_* w_4(\xi),
\]
where $\rho_4$ is the mod $4$ reduction of an integral cohomology class and $i_* \colon H^*(M,\ZZ_2) \to H^*(M,\ZZ_4)$ is the induced map from the inclusion $\ZZ_2 \to \ZZ_4$.
\end{thm}

\begin{thm}\label{TCV3Bundles}
For every $W \in H^2(M;\ZZ_2)$ and $P \in H^4(M;\ZZ)$ there exists a $3$-dimensional vector bundle $\eta$ over $M$ with $w_2(\eta) = W$ and $p_1(\eta)=P$ if and only if
\[
  \rho_4\, P = \mathcal P W
\]
\end{thm}

\begin{rem}\label{ROnTheExistenceOf3BundlesOver5Complex}
Theorem \ref{TCV3Bundles} was also proved by Woodward in \cite[p. 514]{Woodward1982} and by Antieau/ Williams in \cite[Theorem 1]{AW2014}. Note that the minus sign in equation (2) of \cite[Theorem 1]{AW2014} is not correct. It should be either stated without the minus sign (which yields the same equation as in Theorem \ref{TCV3Bundles}) or the authors should use the first Chern class of the complexified bundle and keep the minus sign (since the first Chern class is the negative of the first Pontryagin class). Compare also \cite[p. 514]{Woodward1982} or \cite[Theorem 2]{Cadek1993} for the correct signs.
\end{rem}

Furthermore in \cite{ABBF11} it was proven 

\begin{thm}\label{TABBFStructure}
A $5$-manifold $M$ admits an irreducible $\SOg(3)$-structure if and only if there exists a three-dimensional oriented vector bundle $\eta$ over $M$ such that the tangent bundle is isomorphic to the bundle of symmetric trace-free endomorphism of $\eta$. 
\end{thm}

Combining these three theorems leads us to the

\begin{proof}[\bf Proof of Theorem \ref{TItroduction} (b)]
Suppose first that $p_1(M)$ is divisible by $5$ and $w_4(M)=0$. Then there is a $P \in H^4(X;\ZZ)$ such that $p_1(M)=5\cdot P$. Moreover by Theorem \ref{TCV5Bundles} we have the relation $\rho_4 p_1(M)=\mathcal P w_2(M)$. Hence with coefficients in $\ZZ_4$ we obtain
\[
 \rho_4 P = 5\cdot \rho_4 P = \rho_4 5 \cdot P = \rho_4 p_1(M) = \mathcal P w_2(M).
\]
Hence by Theorem \ref{TCV3Bundles} there exists a vector bundle $\eta$ of rank $3$ such that $p_1(\eta)=P$ and $w_2(\eta)=w_2(M)$. For the induced bundle of symmetric trace-free endomorphisms of $\eta$, $\zeta:=\Sym_0(\eta)$ we have
\[
 p_1(\zeta)=p_1(M),\quad w_2(\zeta)=w_2(M), \quad w_4(\zeta)=0
\]
(see \cite[Theorem 3.2]{ABBF11}). By Theorem \ref{TCV5Bundles} we have that $\zeta$ is isomorphic to the tangent bundle of $M$ and with Theorem \ref{TABBFStructure} this proves that $M$ admits an irreducible $\SOg(3)$-structure.
Now let $M$ admit an irreducible $\SOg(3)$-structure. Then by Proposition \ref{PNeccessaryConditions} we conclude $p_1(M)$ is divisible by $5$ and $w_4(M)=0$.
\end{proof}


\begin{rem}\label{RDrawback}
We believe that in the general case (i.e. without the assumption on $H^4(M;\ZZ)$) the theorem should be true anyway. It is also reasonable to ask if a similar approach as in \cite{Atiyah1970} could be used for the existence of an irreducible $\SOg(3)$-structure.
\end{rem}

\section{Examples}

In \cite{Bobienski2007} the authors classified homogeneous manifolds with irreducible $\SOg(3)$-structures and in \cite{ABBF11} there were also some non-homogeneous examples mentioned as circle bundles over complex surfaces. In \cite{ABBF11} it was shown that $\SUg(3)/\SOg(3)$ does not possess a standard $\SOg(3)$-structure but possesses an irreducible one with $\hat\chi(M)=0$. So it was conjectured that $\hat\chi(M)$ is the second obstruction. However Proposition \ref{PCounterExampleToFriedrich} is counterexample.

\begin{proof}[\bf Proof of Proposition \ref{PCounterExampleToFriedrich}]
Let $\Sigma_d$ be the zero set of the homogeneous polynomial 
\[
    p_d(X_0,X_1,X_2,X_3)=X_0^d + X_1^d + X_2^d + X_3^d
\]
 in $\CC\PP^3$ for $d \in \NN$ and $d\neq0$. $\Sigma_d$ is a $4$-dimensional submanifold which is simply connected by the Lefschetz hyperplane section theorem. Let $u \in H^2(\CC\PP^3;\ZZ)$ be such that $H^*(\CC\PP^3;\ZZ)=\ZZ[u]/(u^4)$. We denote by the same letter $u$ the restriction of the generator of $H^*(\CC\PP^3;\ZZ)$ to $\Sigma_d$. The following facts about $\Sigma_d$ are well known 
\begin{enumerate}
 \item[(a)] The group $H^2(\Sigma_d;\ZZ)$ has no torsion and its rank is equal to $\chi(\Sigma_d)-2= (6-4d +d^2)d-2$,
 \item[(b)] the first and second Chern classes are given by $c_1(\Sigma)=(4-d)u$ and $c_2(\Sigma)=(6-4d+d^2)u^2$ respectively,
 \item[(c)] for the first Pontryagin class we have $p_1(\Sigma)=(4-d^2)u^2$ and for the second Stiefel-Whitney class one obtains $w_2(\Sigma)=d\cdot u \mod 2$.
\end{enumerate}

Let us consider the case $d=3$. We set $\Sigma:=\Sigma_3$ and let $\beta$ denote the intersection form of $\Sigma$. It is known that this space is diffeomorphic to $\CC\PP^2 \# 6\overline{\CC\PP^2}$. Furthermore let $w \in H^2(\Sigma;\ZZ)$ such that $w$ is orthogonal to $u$ with respect to $\beta$ and $w\neq u$, i.e. $w\smile u =0$. Let $\pi \colon M\to \Sigma$ be the  $S^1$-bundle over $\Sigma$ associated to the class $c:=u + w$. Using the Gysin-sequence for sphere bundles one obtains (see also Remark 3.3 in \cite{ABBF11})
\begin{enumerate}
 \item[(a)] $k(M)=\hat\chi(M) = \dim_\RR H^2(\Sigma;\RR) \mod 2$, hence $k(M)=\hat\chi(M)=1$,
 \item[(b)] $M$ is not spin, since this is only the case if $c \equiv w_2(\Sigma) \mod 2$,
 \item[(c)] we have $c\smile 5u = 5u^2 + u\smile w=5u^2=p_1(\Sigma)$, hence by the Gysin-sequence $\pi^*(p_1(\Sigma))=0$ thus $p_1(M)=0$,
 \item[(d)] we also have $w_4(M)=\pi^*(w_4(\Sigma))$. But $w_4(\Sigma)$ is the second Chern class $c_2(\Sigma) \mod 2$ and therefore $w_4(\Sigma) \equiv u^2 \mod 2$. Again we have $c \smile 3u = c_2(\Sigma)$ hence by naturality we obtain $w_4(M)=\pi^*(w_4(\Sigma))=0$. 
\end{enumerate}
This shows that $M$ is a non spin manifold of dimension $5$ with $p_1(M)=0$ and $w_4(M)=0$. Moreover $H^4(M;\ZZ)$ has no $4$-torsion which can be seen as follows: by the Gysin-sequence we have that $H^4(M;\ZZ)$ is isomorphic to $H^4(\Sigma;\ZZ)\cong\ZZ$ modulo the image of the map $H^2(\Sigma;\ZZ) \to H^4(\Sigma;\ZZ)$, $x\mapsto c\smile x$. We can choose $w$ in such a way that the image is $3\ZZ$, hence $H^4(M;\ZZ)=\ZZ_3$ which means, that $M$ can not be simply connected and furthermore that $H^4(M;\ZZ)$ has no $4$-torsion. So by Theorem \ref{TItroduction} (b) $M$ has to admit an irreducible $\SOg(3)$-structure with $\hat\chi(M)=1$.
\end{proof}

\begin{proof}[\bf Proof of Proposition \ref{PConnectedSumsOfSpin5Manifolds}]
If $M_1$ and $M_2$ are spin $5$-manifolds then, the connected sum $M_1\# M_2$ is again spin. The Kervaire semi-characteristic is computed by
\[
 k(M_1 \# M_2) = k(M_1) + k(M_2) + 1 \mod 2.
\]
Furthermore the first Pontryagin and the fourth Stiefel-Whitney class are additive under the operation of building connected sums. Hence we obtain
\begin{eqnarray*}
 k(\#^{2k+1}_{i=1} M_i) &=& \sum_{i=1}^{2k+1} k(M_i) + 2k \mod 2\\
	  &=& \sum_{i=1}^{2k+1} k(M_i) \mod 2.
\end{eqnarray*}
Now since $M_i$ is a spin $5$--manifold with an irreducible $\SOg(3)$--structure it admits a standard $\SOg(3)$--structure (cf. Corollary \ref{CT14aBetterFormulated}), hence by Theorem \ref{TAtiyah} $k(M_i)=0$ for all $i=1,\dots,2k+1$. Thus $k(\#^{2k+1}_{i=1} M_i)=0$ and $p_1(\#^{2k+1}_{i=1} M_i)$ is divisible by $5$, so $\#^{2k+1}_{i=1} M_i$ admits an irreducible $\SOg(3)$--structure by Corollary \ref{CT14aBetterFormulated}.
\end{proof}
%
%
%
%

\begin{exmpl}\label{EMtimesN} Let $N$ be a closed $3$-manifold and $\Sigma$ a closed surface both orientable. We will show that connected sums of $N \times \Sigma$ admit an irreducible $\SOg(3)$-structure. First note that both manifolds are spin. Moreover Steenrod showed that the tangent bundle of $N$ is always parallelizable. Hence $M = N \times \Sigma$ admits an irreducible $\SOg(3)$-structure. By Proposition \ref{PConnectedSumsOfSpin5Manifolds} the manifolds $\#_{i=1}^{2l+1}(N_i \times \Sigma_i)$ have an irreducible $\SOg(3)$-structure where $N_i$ and $\Sigma_i$ are orientable closed manifolds of dimension $3$ and $2$ respectively.
\end{exmpl}

\begin{exmpl}
There are two $S^3$ bundles over $S^2$. One is the trivial bundle covered by Example \ref{EMtimesN} and the other we denote by $\pi \colon M \to S^2$. In the following we will prove that $M$ admits an irreducible $\SOg(3)$-structure. $M$ is simply connected with $H_2(M)=\ZZ$, $H^4(M;\ZZ)=0$ and $w_2(M)\neq 0$. Let $\tau_M$ and $\tau_{S^2}$ be the tangent bundle of $M$ and $S^2$ respectively. Then $\tau_M = \eta \oplus \tau_{S^2}$ with $\eta$ the vertical distribution of $\pi \colon M \to S^2$. It follows that $0\neq w_2(M) = w_2(\eta)$ and $p_1(M)=p_1(\eta)=0$. Hence by Theorem \ref{TCV3Bundles}, \ref{TCV5Bundles} and \ref{TABBFStructure} we have that $\tau_M \cong \Sym_0(\eta)$. By Corollary \ref{CNonSpinCase} (b) $M$ admits an irreducible $\SOg(3)$-structure.
\end{exmpl}

\begin{rem}
Note that every fibre bundle $\pi \colon M \to S^3$ is trivial if $M$ is orientable and of dimension $5$, since $\pi_2$ of the diffeomorphism group of an orientable surface always vanishes.
\end{rem}


%
\nocite{*}
\bibliography{tiso3str}
\bibliographystyle{alpha}

\end{document}